\documentclass{amsart}
\usepackage{amsmath, amssymb,verbatim}
\newtheorem{thm}{Theorem}
\newtheorem{prop}{Proposition}
\newtheorem{Def}{Definition}
\newtheorem{lem}{Lemma}
\newtheorem{rem}{Remark}
\begin{document}
	\title[On the Commutativity of Toeplitz Operators]{On the Commutativity of a certain class of Toeplitz operators}
	
	\date{\today}

\author{Hashem AlSabi and Issam Louhichi}
\address{American University of Sharjah\\
	Department of Mathematics \& Statistics\\
	P.O.Box 26666, Sharjah, UAE.}
\email{b00054409@aus.edu}
\email{ilouhichi@aus.edu}


\begin{abstract}
 In this paper we prove that if the polar decomposition of a symbol $f$ is truncated above, i.e., $f(re^{i\theta} )=\sum_{k=-\infty}^Ne^{ik\theta} f_k (r)$  where the $f_k$'s are radial functions, and if the associated Toeplitz operator $T_f$ commutes with  $T_{z^2+\bar{z}^2}$, then $T_f=Q(T_{z^2+\bar{z}^2})$ where $Q$ is a polynomial of degree at most $1$. This gives a partial answer to an open problem by S. Axler, \u{Z}. \u{C}u\u{c}kovi\'{c} and N. V. Rao \cite[p.~1953]{acr}.
\end{abstract}
\maketitle{}

\section{Introduction}

Let $\mathbb{D}$ be the unit disk of the complex plane $\mathbb{C}$, and $dA=rdr\frac{d\theta}{\pi}$, where $(r,\theta)$ are polar coordinates, be the normalized Lebesgue measure, so that the area of $\mathbb{D}$ is one.  We define the analytic Bergman space, denoted $L^2_a(\mathbb{D})$, to be the set of all analytic functions on $\mathbb{D}$ that are square integrable with respect to the measure $dA$. It is well know that $L^2_a(\mathbb{D})$ is a closed subspace of the Hilbert space $L^2(\mathbb{D},dA)$ and has the set $\{\sqrt{n+1}z^n\ |\ n\geq 0\}$ as an orthonormal basis (see \cite{HH}). Thus, $L^2_a(\mathbb{D})$ is itself a Hilbert space with the usual inner product of $L^2(\mathbb{D},dA)$. Moreover the orthogonal projection, denoted $P$, from $L^2(\mathbb{D},dA)$ onto $L^2_a(\mathbb{D})$, often called the Bergman projection, is well defined. Let $f$ be a bounded function on $\mathbb{D}$. We define on $L_a^2(\mathbb{D})$ the Toeplitz operator $T_f$ with symbol $f$ by $T_f(u)=P(fu),$ for any $u\in L^2_a(\mathbb{D})$.

A natural question to ask is under which conditions is the product (in a sense of composition) of two Toeplitz operators commutative? In other words, when is $T_fT_g=T_gT_f$ for given two Toeplitz operators $T_f$ and $T_g$? It is easy to see from the definition of Toeplitz operators that if the symbol $f$ is analytic and bounded on $\mathbb{D}$, then $T_f$ is simply the multiplication operator by $f$, i.e., $T_f(u)=fu$ for all $u\in L_a^2(\mathbb{D})$. Thus, any two analytic Toeplitz operators (i.e., Toeplitz operators with analytic symbols) commute with each other. Again from the definition of Toeplitz operators, we have that the adjoint of $T_f$ is $T_{\bar{f}}$ where $\bar{f}$ is the complex conjugate of $f$. It follows that if $f$ is antianalytic (i.e., $\bar{f}$ is analytic), then $T^*$ is the multiplication operator by $\bar{f}$. Hence, if two symbols $f$, and $g$ are antianalytic, then their associated Toeplitz operators commute since their adjoints commute. This situation in which the symbols are both analytic (resp. antianalytic) is known to us as the trivial situation. One might ask what if the symbols were harmonic but not necessarily analytic or antianalytic. The answer to this question was given by S. Axler and \u{Z}. \u{C}u\u{c}kovi\'{c} in \cite{ac}. They proved the following:
\begin{thm}[Axler \& \u{C}u\u{c}kovi\'{c}]
If $f$ and $g$ are two bounded harmonic functions in $\mathbb{D}$, then $T_fT_g=T_gT_f$ if and only if
\begin{itemize}
	\item[(a)] both $f$ and $g$ are analytic in $\mathbb{D}$,
	or
	\item[(b)] both $f$ and $g$ are antianalytic in $\mathbb{D}$,
	or
	\item[(c)] $f=\alpha g+\beta$, where $\alpha, \beta$ are constant in $\mathbb{C}$.
\end{itemize}
\end{thm}
So basically if both symbols are harmonic, then the product is commutative only in the trivial case. In fact, the sufficient condition (a) (resp. (b)) says that the operators $T_f$ and $T_g$ (resp. their adjoints $T_{\bar{f}}$ and $T_{\bar{g}}$) are multiplication operators and so they commute. For the sufficient condition (c), since Toeplitz operators are linear with respect to their symbol, we can write $T_f=\alpha T_g+\beta I$ where $I=T_{1}$ is the identity operator on $L^{2}_a(\mathbb{D})$, and hence, since $T_g$ commutes with itself and with the identity, $T_g$ commutes with  $T_f$.

The next natural step was to relax the hypothesis of the previous theorem in order to obtain results for a larger class of symbols. In \cite{acr}, S. Axler, \u{Z}. \u{C}u\u{c}kovi\'{c}, and N. V. Rao proved that analytic Toeplitz operators commute only with other such operators. Their result can be stated as follows:
\begin{thm}[Axler, \u{C}u\u{c}kovi\'{c} \& Rao]
If $g$ is a nonconstant analytic function in $\mathbb{D}$ and if $f$ is bounded in $\mathbb{D}$ such that $T_fT_g=T_gT_f$, then $f$ must be analytic too.	
\end{thm}	
 For Theorem 2, the authors do not ask the function $f$ to be harmonic but only bounded. However this was not without cost. In fact the hypothesis on the symbol $g$ is stronger than the one in Theorem 1 since here $g$ has to be analytic. Finally, the authors conclude \cite{acr} by asking the following open problem: {\bf{"Suppose $g$ is a bounded harmonic function in $\mathbb{D}$ that is neither analytic nor antianalytic. If $f$ is a bounded function in $\mathbb{D}$ such that $T_f$ and $T_g$ commute, must $f$ be of the form $\alpha g+\beta$ for some constants $\alpha, \beta$?"}} The first partial answers to this  problem can be found in \cite{lra} and \cite{y}.

\section{Quasihomogeneous Toeplitz operators}

\begin{Def}
A symbol $f$ is said to be quasihomogeneous of order $p$, and the associated Toeplitz operator $T_f$ is also called a quasihomogeneous Toeplitz operator of order $p$, if $f(re^{i\theta})=e^{ip\theta}\phi(r)$, where $\phi$ is an arbitrary radial function.	
\end{Def}
The motivation behind considering such a family of symbols is that any function $f$ in $L^2(\mathbb{D}, dA)$  has the following polar decomposition (Fourier series)
$$L^2(\mathbb{D}, dA)=\bigoplus_{k\in \mathbb{Z}}e^{ik\theta}\mathcal{R},$$
where $\mathcal{R}=L^2([0,1], rdr)$. In other words $f(re^{i\theta})=\sum_{k\in \mathbb{Z}}e^{ik\theta}f_{k}(r)$, where the $f_k$'s are radial functions in $\mathcal{R}$. So the study of quasihomogeneous Toeplitz operators will allow us to obtain interesting results about Toeplitz operators with more general symbols.

Another interesting property of a quasihomogeneous operator is that it acts on the elements $z^n$ of the orthogonal basis of $L^2_a(\mathbb{D})$ as a shift operator with weight. In fact, if $k\in \mathbb{Z_+}$ (the case where $k$ is a negative integer can be done in the exact same way) and $f_k$ is a bounded radial function, then for any $n\geq 0$ we have
\begin{eqnarray*}
T_{e^{ik\theta}f_k}(z^n)&=&P(e^{ik\theta}f_kz^n)=\sum_{j\geq 0}(j+1)\langle e^{ik\theta}f_kz^n, z_j\rangle z_j\\
&=& \sum_{j\geq 0}(j+1)\left(\int_0^1\int_0^{2\pi}f_k(r)r^{n+j}e^{i(n+k-j)\theta}\frac{d\theta}{\pi}rdr\right)z^j,\\
\end{eqnarray*}
Now, since
$$\int_{0}^{2\pi}e^{i(k+n-j)\theta}\frac{d\theta}{\pi}=\left\{\begin{array}{ll}
0 & \textrm{ if } k+n\neq j\\
2 & \textrm{ if } k+n=j
\end{array}\right.$$
we obtain that
\begin{equation}\label{weight}
T_{e^{ik\theta}f_k}(z^n)=2(n+k+1)\int_{0}^1f_k(r)r^{2n+k+1}dr z^{n+k}.
\end{equation}
The integral $\displaystyle{\int_{0}^1f_k(r)r^{2n+k+1}dr}$ that appears in the weight is known as the Mellin transform.
\begin{Def}
	We define the Mellin transform of a function $\phi$ in $L^1([0,1],rdr)$, denoted $\widehat{\phi}$, to be
	$$\widehat{\phi}(z)=\int_{0}^{1}\phi(r)r^{z-1}dr,\textrm{ for } \Re z\geq 2.$$
\end{Def}
It is well known that the Mellin transform is related to the Laplace transform via the change of variable $r=e^{-u}$. Moreover, for $\phi\in L^1([0,1],rdr)$, $\widehat{\phi}$ is bounded in the right-half plane $\{z\in\mathbb{C} | \Re z\geq 2\}$ and analytic in $\{z\in\mathbb{C} | \Re z> 2\}$.

Using the Mellin transform, we can rewrite Equation (\ref{weight}) as follows
$$T_{e^{ik\theta}f_k}(z^n)=2(n+k+1)\widehat{f_k}(2n+k+2) z^{n+k}.$$
Therefore, we can summarize the above calculation in the following lemma which we shall be using often.
\begin{lem}\label{Mellin}
	Let $k\in\mathbb{Z}$ and $n\in \mathbb{N}$ be two integers, and let $\phi$ be a bounded radial function in $\mathbb{D}$. If $k\geq 0$, then
	$$T_{e^{ik\theta}\phi} (z^n)=2(n+k+1)\widehat{\phi}(2n+k+2)z^{n+k},$$
		and if $k<0$, then
		$$T_{e^{ik\theta}\phi}(z^n)=\left\{\begin{array}{ll}
		0&\textrm{ if } n<|k|\\
		2(n+k+1)\widehat{\phi}(2n+k+2)z^{n+k}& \textrm{ if } n\geq |k|\end{array}\right.$$
\end{lem}

The Mellin transform is going to play a major role in our arguments for the proofs. In fact a function is well determined by its Mellin transform on any arithmetic sequence. We have the following important lemma that can be found in \cite[Remark 2.p~1466]{l}
\begin{lem}\label{complex}
	If $\phi\in L^1([0,1],rdr)$ is such that $\widehat{\phi}(a_n)=0$, where $(a_n)_n$ is a sequence of integers satisfying the condition $\sum_{n}\frac{1}{a_n}=\infty$, then $\widehat{\phi}(z)=0$ on $\{z\in\mathbb{C}|\Re z>2\}$, and therefore $\phi$ is the zero function.
\end{lem}
In other words, the lemma is saying that the Mellin transform is injective, and so two functions whose Mellin transforms coincide on an arithmetic sequence will be equal to each other.

Another classical lemma which we shall use often can be stated as follows:
\begin{lem}\label{periodic}
If $H$ is a bounded analytic function in $\{z\in\mathbb{C}|\Re z>2\}$ such that $H(z+p)=H(z)$, i.e., $H$ is $p$-periodic, then $H$ must be constant.
\end{lem}

When dealing with the product of quasihomogeneous Toeplitz operators, we are often  confronted with the Mellin convolution of the radial functions in their quasihomogeneous symbols. We define the Mellin convolution of two radial functions $\phi$ and $\psi$ in $L^1([0,1],rdr)$, denoted $\phi\ast_M\psi$,  to be
$$\left(\phi\ast_M\psi\right)(r)=\int_{r}^{1}\phi\left(\frac{r}{t}\right)\psi(t)\frac{dt}{t}.$$
It is well known that the Mellin transform converts the Mellin convolution into a product of Mellin transforms. In fact
\begin{equation}\label{conv}
\left(\widehat{\phi\ast_M\psi}\right)(r)=\widehat{\phi}(r)\widehat{\psi}(r),
\end{equation} and so if $\phi$ and $\psi$ are in $L^1(\mathbb{D},dA)$, then so is $\phi\ast_M\psi$. We are now ready to present our main result.

\section{Commutant of $T_{z^2+\bar{z}^2}$}

In this section we shall extend the work started in \cite{lra} and \cite{y}. We consider the Toeplitz $T_{z^2+\bar{z}^2}$ (the symbol $z^2+\bar{z}^2$ is harmonic but neither analytic nor antianalytic). It is known to us that such operator raised to any power $n\geq 2$ is not a Toeplitz operator. We shall prove that if the symbol $f$ has truncated polar decomposition i.e., $f(re^{i\theta})=\sum_{k=-\infty}^{N}e^{ik\theta}f_k(r)$ where $N$ is a positive integer,  and if $T_f$ commutes with $T_{z^2+\bar{z}^2}$, then $T_f$ is polynomial of degree at most one in $T_{z^2+\bar{z}^2}$. This result goes in the direction of the open problem we mentioned previously. We would like to emphasize the fact that though we are using the same tools and techniques as in \cite{lra}, new ideas and tricks  were needed to overcome numerous obstacles we faced in the proof of the main result.

In our presentation of the main theorem, we  shall proceed as follows: First we prove that if $f(re^{i\theta})=\sum_{k=-\infty}^{N}e^{ik\theta}f_k(r)$ is such that $T_f$ commutes with $T_{z^2+\bar{z}^2}$, then $N$ has to be an even number. Second, we shall demonstrate that this same $N$ cannot exceed $4$. Finally, we shall exhibit all the radial functions $f_k$ for $k\leq 4$, and shall show that $f_k(r)=0$ for $k\neq\{-2,0,2\}$, $f_k(r)=cr^2$ for $k=\{2,-2\}$, and $f_0(r)=c_0$ where $c, c_0$ are constants. Hence, by reconstructing the symbol $f$, we shall obtain that $$f(re^{i\theta})=ce^{2i\theta}r^2+c_0+ce^{-2i\theta}r^2=cz^2+c_0+c\bar{z}^2,$$ and therefore $T_f=cT_{z^2+\bar{z}^2}+c_0I$.
\begin{prop}\label{}
Let $N$ be a positive odd integer. If $ f(re^{i\theta})= \sum_{k=-\infty}^{N} e^{ik\theta}f_k(r)$ is a nonzero symbol such that $ T_f $ commutes with $ T_{z^2+\bar{z}^2} $ , then $ f_N(r)=0$. In other words, $f$ is of the form $ f(re^{i\theta})= \sum_{k=-\infty}^{M} e^{ik\theta}f_k(r)$ where $ M $ is even.
\end{prop}

\begin{proof}
If $ T_f $ commutes with $ T_{z^2+\bar{z}^2} $ , then $$ T_{z^2+\bar{z}^2}T_f (z^n)=T_fT_{z^2+\bar{z}^2}(z^n),  \  \forall n\geq0,$$ or
$$ \left(\sum_{k=-\infty}^{N} T_{e^{ik\theta}f_k}\right)T_{z^2+\bar{z}^2}(z^n)=  T_{z^2+\bar{z}^2}\left(\sum_{k=-\infty}^{N}T_{e^{ik\theta}f_k}\right)(z^n), \ \forall  n\geq 0. $$
In the above equation, the term with the highest degree is $ z^{n+N+2} $. It comes on the left hand side from the product $ T_{e^{iN\theta}f_N} T_{z^2}(z^n) $ only,
and on the right hand side from the product $ T_{z^2}T_{e^{iN\theta}f_N}(z^n) $ only.
Thus, by equality, we must have  $$ T_{e^{iN\theta}f_N} T_{z^2}(z^n)= T_{z^2}T_{e^{iN\theta}f_N}(z^n), \ \forall  n\geq 0.  $$
	Since $ z^2 $ is analytic,  $ e^{iN\theta}f_N $ must be analytic too. Which is possible if and only if $ f_N=c_Nr^N $, i.e., $ e^{iN\theta}f_N = c_Nz^N $.
	
	Redoing the same argument for the term in $z$ of degree $N+n-2$, we obtain
	$$ c_N T_{z^N} T_{\bar{z}^2}(z^n) + T_{e^{i(N-4)\theta}f_{N-4}} T_{z^2} (z^n) = c_N T_{\bar{z}^2} T_{z^N}(z^n) + T_{z^2} T_{e^{i(N-4)\theta}f_{N-4}} (z^n), \ \forall  n\geq 0, $$
	which, using Lemma \ref{Mellin}, is equivalent to
\begin{eqnarray*}
	 c_N \frac{n-1}{n+1} + 2(n+N-1)\hat{f}_{N-4}(2n+N+2) &=& c_N \frac{n+N-1}{n+N+1} + \\
	& & 2(n+N-3)\hat{f}_{N-4}(2n+N-2),
\end{eqnarray*}
	for all $n\geq 2$. Thus, Lemma \ref{complex} implies
\begin{multline}\label{lemma}
	(z+2N-2)\hat{f}_{N-4}(z+N+2) - (z+2N-6)\hat{f}_{N-4}(z+N-2)= c_N \left[\frac{z+2N-2}{z+2N+2} - \frac{z-2}{z+2}\right],
\end{multline}
	for $\Re z\geq 4$. Now, we introduce the function
	\begin{equation}\label{f_*}
	 f_* (r)= -4c_N r^2 \frac{1-r^{2N}}{1-r^4}.
	\end{equation}
	By direct calculation and simple algebraic operations, we can see that
	\begin{eqnarray*}
	\widehat{f_*}(z+4)-\widehat{f_*}(z)&=& \int_{0}^{1} -4c_N (r^4-1)\frac{1-r^{2N}}{1-r^4} r^{z-1} dr\\ &=&
	 4c_N\left[\frac{1}{z+2}-\frac{1}{z+2N+2}\right]\\
	 &=& c_N \left[\frac{z+2N-2}{z+2N+2} - \frac{z-2}{z+2}\right].
	\end{eqnarray*}
	We denote by $ F(z)=(z+2N-6)\widehat{r^{N-2}f_{N-4}}(z) $ and $ G(z)=\widehat{f_*}(z) $. Then Equation (\ref{lemma}) can be rewritten as $$ F(z+4)-F(z) = G(z+4)- G(z),$$
	and so Lemma \ref{periodic} implies
	$$F(z)= c_{N-4} + G(z),\textrm{ for some constant } c_{N-4},$$ or
	$$ (z+2N-6)\widehat{r^{N-2}f_{N-4}}(z) = c_{N-4} + \widehat{f_*}(z).$$
	Since $\frac{1}{z+2N-6}=\widehat{r^{2N-6}}(z)$, the above equation becomes
	\begin{equation}\label{N}
	\widehat{r^{N-2}f_{N-4}}(z) = c_{N-4}\widehat{r^{2N-6}}(z) + \widehat{r^{2N-6}}(z)\widehat{f_*}(z).
	\end{equation}
	Since by Equation (\ref{conv}), $\widehat{r^{2N-6}}(z)\widehat{f_*}(z)= \widehat{(r^{2N-6}*_M f_*)}(z)$, we have
	$$f_{N-4}(r)=c_{N-4}r^{N-4}+\frac{1}{r^{N-2}}\left(r^{2N-6}\ast_{M}f_{*}\right)(r).$$
 Let us denote by
	$$I_N=\left(r^{2N-6}*_M f_*\right)(r)
	= -4c_N \int_{r}^{1} \frac{r^{2N-6}}{t^{2N-6}} t^2 \frac{1-t^{2N}}{1-t^4} \frac{dt}{t}. $$
	Next, we need to determine the conditions on $N$ under which
		$ \frac{1}{r^{N-2}} I_N$, as a function of $r$, is in $L^1([0,1], rdr)$. Otherwise $c_N$ must be zero, in which case $f_N(r)=0$ and we are done.
	Since $$\frac{1}{r^{N-2}} I_N = r^{N-4} \int_{r}^{1} \frac{1}{t^{2N-7}} \frac{1-t^{2N}}{1-t^4} dt \geq r^{N-4} \int_{r}^{1} \frac{1}{t^{2N-7}} dt, \textrm{  whenever }\  N\geq3, $$
	it follows that
	$$ \frac{1}{r^{N-2}} I_N \geq \frac{1}{8-2N} \left(r^{N-4} - r^{4-N}\right). $$
	Now the function on the right hand side of the above inequality is in $ L^1 ([0,1], rdr) $ if and only if
	$$ "N-4+1 \geq 0\textrm{ and }4-N+1\geq 0",\textrm{ i.e., }   3\leq N \leq 5.$$
	But since $ N  $ is an odd positive integer, we have either $N=3$ or $N=5$. We recall that the previous inequality was obtained after the assumption $ N\geq 3 $, and
	so we shall look at the case where $ N=1 $ separately.
	
	\begin{itemize}
	\item[{\bf{Case $N=1$.}}] If we set $ N=1 $, then
	$$ I_1 = -4c_1 \int_{r}^{1} \frac{r^{-4}}{t^{-4}}t^2\frac{1-t^2}{1-t^4} \frac{dt}{t} = \frac{-4c_1}{r^4} \left[\frac{4\ln2-2}{8}- \frac{r^4}{4}+\frac{r^2}{2}-\frac{\ln(1+r^2)}{2}\right], $$ and we have
	$$ f_{-3}(r)=\frac{c_{-3}}{r^3} +  \frac{-4c_1}{r^3} \left[\frac{4\ln2-2}{8}- \frac{r^4}{4}+\frac{r^2}{2}-\frac{\ln(1+r^2)}{2}\right].$$
	Now, it is easy to see that $f_{-3}$ is in $L^1([0,1], rdr)$ if and only if $c_{-3}=0$ and $c_1=0$. Hence $f_1(r)=c_1r=0$.
	\item[{\bf{Case $N=3$.}}] If we let $ N=3 $, then
	the terms in $ z^{n+1} $ comes from the following equality:
	$$ c_3 T_{z^3} T_{\bar{z}^2}(z^n) + T_{e^{-i\theta}f_{-1}} T_{z^2} (z^n) = c_3 T_{\bar{z}^2} (z^n) + T_{z^2} T_{e^{-i\theta}f_{-1}} (z^n), \forall n\geq 0. $$
	In particular, for $ n=1$ we have
	\begin{equation}\label{c_1}
	6 \widehat{f_{-1}}(7) z^2 = 2\widehat{f_{-1}}(3) z^2\textrm{ i.e., }3\widehat{f_{-1}}(7)=\widehat{f_{-1}}(3).	
	\end{equation}
	Since 	$\widehat{f_{-1}}(3)=\widehat{r f_{-1}}(2)$, Equation $(\ref{N})$ with $N=3$ implies
	\begin{equation*}
 \widehat{r f_{-1}}(2)= c_{-1} \widehat{1}(2) + \widehat{1}(2) \widehat{f_*}(2),
	\end{equation*}
	where $\widehat{f_*}$ is obtained from $(\ref{f_*})$ with $N=3$, and we have
	\begin{equation*}
	\widehat{f_*}(2) = -4c_3 \int_{0}^{1} r^2\frac{1-r^{6}}{1-r^4}r dr = -4c_3 \left(\frac{2-3\ln2}{3}\right);
	\end{equation*}
	and similarly
	\begin{equation*}
	\widehat{f_{-1}}(7)=\widehat{r f_{-1}}(6) = c_{-1} \widehat{1}(6) + \widehat{1}(6) \widehat{f_*}(6),
	\end{equation*}
	with
	\begin{equation*}
	\widehat{f_*}(6) = -4c_3 \int_{0}^{1} r^2\frac{1-r^{6}}{1-r^4}r^{5}dr = -c_3 \left(\frac{31-30\ln2}{15}\right).
	\end{equation*}
	Therefore, $(\ref{c_1})$ implies
	\begin{equation*}
	\frac{1}{2} c_{-1} -\left(\frac{31-30\ln2}{30}\right) c_3 = \frac{1}{2}c_{-1}-2\left(\frac{2-3\ln2}{3}\right)c_3,
	\end{equation*}
	which is possible if and only if $c_{3}=0$, and hence $f_3(r)=0$.
	\item[{\bf{Case $N=5$.}}] If we set $ N=5 $, then
	the terms in $z$ of degree $n+3$ comes from the following equation:
	\begin{equation*}
		 c_5 T_{z^5} T_{\bar{z}^2} (z^n)  + T_{e^{i\theta}f_1} T_{z^2} (z^n) = c_5 T_{\bar{z}^2} T_{z^5}(z^n) + T_{z^2} T_{e^{i\theta}f_1} (z^n), \forall n\geq 0.
	\end{equation*}
In particular, for $n=0$ and using Lemma \ref{Mellin}, we have
\begin{equation*}
 8\widehat{f_1}(7) z^3 = c_5 \frac{4}{6} z^3 + 4\hat{f_1}(3) z^3,
\end{equation*}
or
\begin{equation} \label{c_5}
8\widehat{f_1}(7)= \frac{2}{3}c_5 + 4\hat{f_1}(3).
\end{equation}
Since $\widehat{f_1}(3)=\widehat{r^3f_1}(0)$, Equation $(\ref{N})$ with $N=5$ implies
$$\widehat{r^3f_1}(0)=c_1\widehat{r^4}(0)+\widehat{r^4}(0)\widehat{f_*}(0),$$
where $\widehat{f_*}$ is obtained from $(\ref{f_*})$ with $N=5$, and we have
$$\widehat{f_*}(0)=-4c_5\int_{0}^{1}r^2\frac{1-r^{10}}{1-r^4}r^{-1}dr=-c_5\left(\frac{3+4\ln 2}{2}\right).$$
Similarly , we have
$$\widehat{f_1}(7)=\widehat{r^3f_1}(4)=c_1\widehat{r^4}(4)+\widehat{r^4}(4)\widehat{f_*}(4),$$
with
$$\widehat{f_*}(4)=-4c_5\int_{0}^{1}r^2\frac{1-r^{10}}{1-r^4}r^3 dr=-c_5\left(\frac{-1+12\ln 2}{6}\right).$$
Now, substituting $\widehat{f_1}(7)$ and $\widehat{f_1}(3)$ in Equation $(\ref{c_5})$, we obtain
$$\left(\frac{1-12\ln 2}{6}\right)c_5=\frac{2}{3}c_5-\left(\frac{3+4\ln 2}{2}\right)c_5,$$
which is true if and only if $c_5=0$, and hence $f_5(r)=0$. This completes the proof.
\end{itemize}	
	
\end{proof}
\begin{prop}
	If $ f(re^{i\theta})= \sum_{k=-\infty}^{N} e^{ik\theta}f_k(r)$ where $ N $ is a positive even integer,  is such that $ T_f $ commutes with $ T_{z^2+\bar{z}^2} $ , then $ N\leq4 $.
	
\end{prop}
\begin{proof}
	If $ T_f $ commutes with $ T_{z^2+\bar{z}^2} $ , then $$ T_{z^2+\bar{z}^2}T_f (z^n)=T_fT_{z^2+\bar{z}^2}(z^n),\  \forall n\geq0$$ or
		$$ \Big(\sum_{k=-\infty}^{N} T_{e^{ik\theta}f_k}\Big)T_{z^2+\bar{z}^2}(z^n)=  T_{z^2+\bar{z}^2}\Big(\sum_{k=-\infty}^{N}T_{e^{ik\theta}f_k}\Big)(z^n), \ \forall  n\geq 0. $$
		In the above equation, the term with the highest degree is $ z^{n+N+2} $.
	On the left-hand side , this term comes from the product $ T_{e^{iN\theta}f_N} T_{z^2}(z^n) $ only, and on the right-hand side it is obtained from the product $ T_{z^2}T_{e^{iN\theta}f_N}(z^n) $ only.
	Thus, by equality, we must have $$ T_{e^{iN\theta}f_N} T_{z^2}(z^n)= T_{z^2}T_{e^{iN\theta}f_N}(z^n),\  \forall  n\geq 0.  $$
	Since $ z^2 $ is analytic, $ e^{iN\theta}f_N $ is analytic too. This is possible if and only if $ f_N=c_Nr^N $ i.e., $ e^{iN\theta}f_N = c_Nz^N $.\\
	Redoing the same argument for the term in $z$ of degree $N+n-2$, we obtain
	$$ c_N T_{z^N} T_{\bar{z}^2}(z^n) + T_{e^{i(N-4)\theta}f_{N-4}} T_{z^2} (z^n) = c_N T_{\bar{z}^2} T_{z^N}(z^n) + T_{z^2} T_{e^{i(N-4)\theta}f_{N-4}} (z^n), \ \forall  n\geq 0, $$
	which, using Lemma \ref{Mellin} , becomes
	\begin{eqnarray*}
	c_N \frac{n-1}{n+1} + 2(n+N-1)\widehat{f_{N-4}}(2n+N+2)  &=& c_N \frac{n+N-1}{n+N+1} \\ &+& 2(n+N-3)\widehat{f_{N-4}}(2n+N-2)
	\end{eqnarray*}
	for $n\geq 2$. Thus, Lemma \ref{complex} implies
	\begin{multline*}
	2(z+N-1)\widehat{f_{N-4}}(2z+N+2) - 2(z+N-3)\widehat{f_{N-4}}(2z+N-2) =  c_N \Big[\frac{z+N-1}{z+N+1}- \frac{z-1}{z+1}\Big]
	\end{multline*}
	for $\Re z\geq 2$. Now,  if we let $F$ and $G$ to be
	$$ F(z)= 2(z+N-3)\widehat{f_{N-4}}(2z+N-2) \textrm{ and } G(z)= c_{N}\sum_{i=0}^{\frac{N}{2}-1} \frac{z-1+2i}{z+1+2i},$$
	the the previous equation can be written as
	$$ F(z+2)-F(z)=G(z+2)-G(z), \textrm{ for } \Re z\geq 2.$$
	Hence, by Lemma \ref{periodic} we have the following,
$$ 2(z+N-3)\widehat{f_{N-4}}(2z+N-2) = c_{N-4} + \sum_{i=0}^{\frac{N}{2}-1} \frac{z-1+2i}{z+1+2i}, $$
 or
 $$\widehat{f_{N-4}}(2z+N-2) =  \frac{c_{N-4} }{2(z+N-3)} + \frac{c_N}{2(z+N-3)} \sum_{i=0}^{\frac{N}{2}-1} \frac{z-1+2i}{z+1+2i}. $$
At this point we shall assume $N\geq 6$ and we shall prove that in this case $f_{N-4}$ will not be in $L^1([0,1],rdr)$ unless $c_N=0$ and so $f_N(r)=0$. Therefore $N$ shall  be strictly less than 6 and because $N$ is even we  shall conclude that $N\leq 4$. If $N\geq 6$, then
		\begin{eqnarray*}
		\widehat{f_{N-4}}(2z+N-2) &= & \frac{c_{N-4} }{2(z+N-3)} +c_N \Big[ \frac{1}{2(z+N-1)} \\
		&+& \frac{z+N-5}{2(z+N-3)^2} + \frac{1}{2(z+N-3)}\sum_{i=0}^{\frac{N}{2}-3} \frac{z-1+2i}{z+1+2i}\Big]\\
     	 &= &  c_{N-4}\widehat{r^{N-4}}(2z+N-2) +c_N \Big[ \widehat{r^N} (2z+N-2)\\ &+& \widehat{r^{N-4}} (2z+N-2) + \widehat{r^{N-4}\ln r} (2z+N-2) \\ &+& \frac{1}{2}\sum_{i=0}^{\frac{N}{2}-3} \frac{1}{z+N-3} +\frac{1}{2} \sum_{i=0}^{\frac{N}{2}-3}\frac{2}{N-4-2i}\Big(\frac{1}{z+N-3}-\frac{1}{z+2i+1} \Big) \Big] \\
 &= &  c_{N-4}\widehat{r^{N-4}}(2z+N-2) +c_N \Big[ \widehat{r^N} (2z+N-2)\\ &+& \widehat{r^{N-4}} (2z+N-2) + \widehat{r^{N-4}\ln r} (2z+N-2) \\ &+& \left(\frac{N}{2}-3\right)\widehat{r^{N-4}}(2z+N-2)\\
 & +& \sum_{i=0}^{\frac{N}{2}-3}\frac{2}{N-4-2i}\Big(\widehat{r^{N-4}}(2z+N-2) - \widehat{r^{4i-N+4}}(2z+N-2) \Big) \Big].
\end{eqnarray*}
Hence,
\begin{eqnarray*}
	f_{N-4}(r)&= & c_{N-4}r^{N-4}+c_N \Big[ r^N +\left(\frac{N}{2}-2\right)r^{N-4}+ r^{N-4}\ln r\\
	& +& \sum_{i=0}^{\frac{N}{2}-3}\frac{2}{N-4-2i}\Big(r^{N-4} - r^{4i-N+4} \Big) \Big].
\end{eqnarray*}
Now, the term $r^{4i-N+4}$ is in $L^1([0,1],rdr)$ if and only if   $ 4i-N+4 \geq-1 $ with $ 0 \leq i \leq \frac{N}{2}-3 $. Otherwise the constant $c_N$ must be zero. In particular, for $ i=0 $,
we must have $ -N+4 \geq -1, \ \text{i.e.,} \ N\leq5 $. Therefore $ N $ cannot be greater than or equal to 6, otherwise $c_N=0$. Since $ N $ is even and $ N < 6  $, we  deduce that  $ N\leq4,$ i.e., $N=4$ or $N=2$
	\end{proof}

We are now ready to state our main result.
	\begin{thm}
		If $ f(re^{i\theta})= \sum_{k=-\infty}^{N} e^{ik\theta} f_k(r) $ is such that $ T_f  T_{z^2+\bar{z}^2} = T_{z^2+\bar{z}^2}  T_f $ then $T_f$ is a polynomial of degree at most one in $T_{z^2+\bar{z}^2}$. In other words, $f(z)=c_2(z^2+\bar{z}^2)+c_0$ where $c_2, c_0$ are constants.
	\end{thm}
\begin{proof}  From the previous propositions we know that $N$ is even and $N\leq 4$. We shall prove that $f_k(r)=0$ for all $k\neq\{-2,0,2\}$, $f_0(r)=c_0$, and $f_2(r)=c_2r^2=f_{-2}(r)$ for some constants $c_0$ and $c_2$.
	
		Since $ T_f $ commutes with $ T_{z^2+\bar{z}^2} $ , we have $$ T_{z^2+\bar{z}^2}T_f (z^n)=T_fT_{z^2+\bar{z}^2}(z^n), \ \forall n\geq0$$ or	
	\begin{equation}
	\label{sum}
	 \Big(\sum_{k=-\infty}^{4} T_{e^{ik\theta}f_k}\Big)T_{z^2+\bar{z}^2}(z^n)= T_{z^2+\bar{z}^2} \Big(\sum_{k=-\infty}^{4} T_{e^{ik\theta}f_k}\Big)(z^n), \ \forall  n\geq 0.
	\end{equation}
	In the equation above, the term in $z$ with the highest degree is $ z^{n+6} $, and it is coming from the product of $ T_{e^{i4\theta}f_4} T_{z^2}(z^n) $ on the left hand side, and from $ T_{z^2} T_{e^{i4\theta}f_4} (z^n) $ on the right hand side.
	Thus, by equality, we must have $$ T_{e^{i4\theta}f_4} T_{z^2}(z^n)= T_{z^2}T_{e^{i4\theta}f_4}(z^n), \ \forall  n\geq 0  $$
	Since $ z^2 $ is analytic,  $ e^{i4\theta}f_4 $ must be analytic as well by Theorem 2, which is possible if and only if  $ f_4(r) = c_4 r^4 $, i.e., $T_{e^{4i\theta}f_4}=c_4T_{z^4}$.
	Next, we shall prove that $ f_{0} = c_{0}  $ and $ f_4(r) = 0 $. In $(\ref{sum})$, the terms in $ z^{n+2} $  come from the following equality
	\begin{eqnarray} \label{eq0}
	c_4 T_{z^4} T_{\bar{z}^2} (z^n) + T_{f_0} T_{z^2} (z^n) = c_4 T_{\bar{z}^2}  T_{z^4} (z^n) + T_{z^2} T_{f_0} (z^n), \ \forall n\geq0
	\end{eqnarray}
	which, using Lemma \ref{Mellin} and Lemma \ref{complex}, is equivalent to
	$$ 2(z+3)\widehat{f_0}(2z+6) - 2(z+1)\widehat{f_0}(2z+2)= c_4 \Big[\frac{z+3}{z+5}-\frac{z-1}{z+1}\Big], \textrm{ for } \Re z\geq 2. $$
	If we let $ F(z)= 2(z+1)\widehat{f_0}(2z+2) $ and $ \displaystyle{G(z)= c_4 \Big[\frac{z-1}{z+1} + \frac{z+1}{z+3}\Big]} $, then the previous equation can be written as
	$$ F(z+2) - F(z)= G(z+2)-G(z).$$
	Hence, Lemma \ref{periodic} implies
	$$ F(z) = c_0 + G(z), \textrm{ for some constant } c_0. $$
	Therefore
	\begin{eqnarray*} \label{f0}
	\widehat{f_0}(2z+2) &=& \frac{c_0}{2(z+1)} + \frac{c_4}{2(z+1)}\Big[\frac{z-1}{z+1}+\frac{z+1}{z+3}\Big] \\&=& \frac{c_0}{2(z+1)} + \frac{c_4}{2}\Big[\frac{1}{z+1} - \frac{2}{(z+1)^2}+ \frac{1}{z+3}\Big].
\end{eqnarray*}
	 Since $\widehat{r^m}(z)=\frac{1}{z+m}$ and $\widehat{r^m\ln r}(z)=-\frac{1}{(z+m)^2}$ for any integer $m$, the above equality becomes
	 \begin{equation}\label{f_0}
	 \widehat{f_0}(2z+2)= (c_0 + c_4) \widehat{1}(2z+2) + c_4 \Big[\widehat{r^4}(2z+2)+4 \widehat{\ln r}(2z+2)\Big].
	 \end{equation}
	Now, if we take $ n=0 $ in Equation (\ref{eq0}) and apply Lemma \ref{Mellin}, we obtain
	\begin{eqnarray} \label{eq1}
	6\widehat{f_0}(6) =  \frac{6c_4}{10}  + 2\widehat{f_0}(2).
	\end{eqnarray}
	Since
	$$ \widehat{f_0}(2) = \frac{c_0 + c_4}{2} - \frac{5c_4}{6}, $$
	and
	$$ \widehat{f_0}(6)= \frac{c_0 + c_4}{6} - \frac{c_4}{10}, $$
	Equation (\ref{eq1}) becomes
	$$ \frac{6c_4}{10} + c_0 + c_4 - \frac{5c_4}{3} = \frac{6c_4}{10} + c_0 + c_4 - \frac{2c_4}{3}.  $$
	But, it is easy to see that the above equality is possible if and only if $ c_4=0 $ and therefore $ f_4 (r)= 0 $, while Lemma \ref{complex} and Equation $(\ref{f_0})$ imply that  $f_0 (r)= c_0$.
	
	Next, we shall prove that $ f_{-4}(r)= 0 $, and consequently $ f_{-4+4k}(r) = 0,\ \forall k\leq -1. $
	In $(\ref{sum})$, the terms $ z^{n-2} $ come from the product of  $T_{f_0}$ with $T_{\bar{z}^2}$   and the product of $ T_{z^2}$ with $T_{e^{-4i\theta}f_{-4}}$. Thus, we must have
	$$  T_{f_0} T_{\bar{z}^2} (z^n)+  T_{e^{-4i\theta}f_{-4}}T_{z^2}(z^n) = T_{\bar{z}^2}T_{f_0}(z^n) + T_{z^2} T_{e^{-4i\theta}f_{-4}}(z^n), \forall n\geq 0. $$
	Using Lemma \ref{Mellin}, we obtain that for $ n\geq4 $
	\begin{eqnarray*}
	 2(n-1)\widehat{f_{-4}}(2n+2) &-& 2(n-3)\widehat{f_{-4}}(2n-2) \\ &=&  2(n-1)\widehat{f_0}(2n+2) - \frac{4(n-1)^2}{2n+2}\widehat{f_0}(2n-2)
	 \\&=& 2(n-1)\frac{c_0}{2n+2}-\frac{4(n-1)^2}{2n+2}\frac{c_0}{2n-2} \\ &=& 0.
	\end{eqnarray*}
	By letting $ F(z)= 2(z-3)\widehat{f_{-4}}(2z-2)$, the previous equation and Lemma \ref{complex} imply $$ F(z+2)-F(z)= 0,\textrm{ for } \Re z\geq 4,$$ and so Lemma \ref{periodic} yields  $$ F(z) = c_{-4}, \textrm{ for some constant } c_{-4}. $$ Thus $$ 2(z-3)\widehat{f_{-4}}(2z-2) = c_{-4},$$
	or $$\widehat{f_{-4}}(2z-2) = \frac{c_{-4}}{2z-6}= c_{-4} \widehat{r^{-4}}(2z-2), \textrm{ for } \Re z\geq 4.$$
	Hence Lemma \ref{complex} implies $f_4(r)=c_{-4}r^{-4}$. But since $r^{-4} \notin L^1 ([0,1], rdr) $, we must have $ c_{-4}=0 $, and therefore $ f_{-4} (r) = 0 $. Now in $(\ref{sum})$, the terms in $ z^{n-6} $ come from $ T_{e^{-8i\theta}f_{-8}} T_{z^2} (z^n)$ and $ T_{z^2}T_{e^{-8i\theta} f_{-8}}(z^n) $ only because $T_{e^{-4i\theta}f_{-4}}=0$, and so  $$ T_{e^{-8i\theta}f_{-8}} T_{z^2}(z^n) = T_{z^2} T_{e^{-8i\theta}f_{-8}}(z^n), \forall n\geq 0,$$
	i.e., $T_{e^{-8i\theta}f_{-8}}  $ commutes with $ T_{z^2} $, and hence by Theorem 2 we conclude that $ f_{-8}(r)=0 $. Similarly, using the same argument, we prove that $ f_{-4+4k}(r)=0$ for all $k\leq -1.$
	
	 Next, we shall prove that $ f_{-1} (r)= 0=f_{-3}(r)$, and consequently $f_{-1+4k}(r)=0$ for all $k\leq -1$. In $(\ref{sum})$, the terms in $ z^{n+5}  $ come only from the product of $ T_{e^{3i\theta} f_3} $ with $ T_{z^2} $. Thus, Theorem 2 implies $ f_3 (r)= c_3 r^3 $ and so $ T_{e^{3i\theta}f_3} = c_3 T_{z^3} $.
	 Similarly, the terms in $ z^{n+1}  $ come from the product of $ c_3T_{z^3}$ with $T_{\bar{z}^2}  $  and the product of $ T_{e^{-i\theta}f_{-1}}$ with $T_{z^2} $. Thus we must have
	 \begin{equation} \label{eq2}
	 c_3T_{z^3} T_{\bar{z}^2}(z^n) + T_{e^{-i\theta}f_{-1}} T_{z^2}(z^n) = c_3 T_{\bar{z}^2}T_{z^3}(z^n) + T_{z^2}T_{e^{-i\theta}f_{-1}}(z^n) \ \forall n\geq0,
	 \end{equation}
	 which, using Lemma \ref{Mellin},  is equal to
	 $$ c_3 \frac{2(n-2)}{2n+2} + 2(n+2)\widehat{f_{-1}}(2n+5) = c_3 \frac{2(n+2)}{2n+8} + 2n \widehat{f_{-1}}(2n+1), \forall n\geq 2, $$
	 or
	 \begin{equation*}
	 2(n+2)\widehat{f_{-1}}(2n+5) - 2n \widehat{f_{-1}}(2n+1) = c_3 \Big[\frac{n+2}{n+4}- \frac{n-1}{n+1}\Big], \forall n\geq 2.
	 \end{equation*}
	 Now, Lemma \ref{complex} implies
	 \begin{equation}\label{eq3}
	 (z+4)\widehat{rf_{-1}}(z+4)-z\widehat{rf_{-1}}(z)=c_3 \Big[\frac{z+4}{z+8} - \frac{z-2}{z+2}\Big], \textrm{ for } \Re z\geq 4.
	 \end{equation}
	 Here, we introduce a new function $ \displaystyle{f_* (r) = -4c_3 r^2 \frac{1-r^6}{1-r^4} }$, which clearly is in $L^1([0,1], rdr)$.
	 By a direct calculation of the Mellin transform of $(r^4-1)f_*$, we obtain
	 \begin{eqnarray*}
	 \widehat{(r^4 - 1)f_*}(z) &=& 4c_3 \int_{0}^{1} (1-r^4) \frac{1-r^6}{1-r^4} r^{z+1}dr \\ &=& c_3 \Big[\frac{4}{z+2} - \frac{4}{z+8}\Big] \\&=& c_3 \Big[\frac{z+4}{z+8} - \frac{z-2}{z+2}\Big].
	 \end{eqnarray*}
	Thus Equation $(\ref{eq3})$ can be written as
	  \begin{equation*}
	 (z+4)\widehat{rf_{-1}}(z+4) - z \widehat{f_{-1}}(z) = \widehat{(r^4 - 1)f_*}(z)=\widehat{r^4f_*} (z)-\widehat{f_*}(z)=\widehat{f_*}(z+4)-\widehat{f_*}(z).
	  \end{equation*}
	  If we let $F(z)=z\widehat{rf_{-1}}(z)$, then the equation above if simply  $$ F(z+4) - F(z) = f_*(z+4)-f_*(z), $$  and so Lemma \ref{periodic} implies
	  $$ F(z) = c_{-1} + f_*(z) ,\textrm{ for some constant } c_{-1}. $$
	  Thus
	  \begin{equation} \label{eq4}
	  r\widehat{f_{-1}}(z) = \frac{c_{-1}}{z} + \frac{\widehat{f_*}(z)}{z} = c_{-1}\widehat{1}(z) + \widehat{1}(z)\widehat{f_*}(z).
	  \end{equation}
  Now, using the Mellin convolution property $(2)$, we have
  \begin{eqnarray*}
  	 \widehat{1}(z)\widehat{f_*}(z) &=& \widehat{1 *_{M} f_*}(z)\\ &=& \int_{r}^{1} 1(\frac{r}{t}) f_*(t) \frac{dt}{t} \\ &=& -4c_3 \int_{r}^{1} t^2 \frac{1-t^6}{1-t^4} \frac{dt}{t}\\ &=& -4c_3 \Big[\Big(\frac{1}{2}+ \frac{\ln 2}{2}\Big)-\Big(\frac{r^4}{4}+ \frac{\ln (1+r^2)}{2}\Big)\Big].
\end{eqnarray*}
Hence, Equation (\ref{eq4}) and Lemma \ref{complex} imply
	$$  rf_{-1}(r) = c_{-1} -4c_3 \Big[\Big(\frac{1}{2}+ \frac{\ln 2}{2}\Big)-\Big(\frac{r^4}{4}+ \frac{\ln (1+r^2)}{2}\Big)\Big], $$ or
	$$ f_{-1}(r) = \frac{c_{-1}}{r} - c_3 \Big[\Big(\frac{2+2\ln 2}{r}\Big) - \Big(r^3 + \frac{2\ln (1+r^2)}{r}\Big)\Big]$$
	Now in Equation $(\ref{eq2})$, if we set $n=0$ and apply Lemma \ref{Mellin}, we obtain
	\begin{equation}\label{f_{-1}}
	4\widehat{f_{-1}}(5)=4c_3\widehat{r^2}(6),
	\end{equation}
	 and so Equation $(\ref{eq4})$ implies
	$$\widehat{f_{-1}}(5)=\widehat{rf_{-1}}(4)=c_{-1}\widehat{1}(4)+\widehat{1}(4)\widehat{f_*}(4),$$ where
	$$ \widehat{f_*}(4)= -4c_3\int_{0}^{1} r^2 \frac{1-r^6}{1-r^4}r^3dr = c_3 \Big(\frac{1}{2}-2\ln 2 \Big). $$
	Thus Equation $(\ref{f_{-1}})$ becomes
	$$ \frac{c_{-1}}{4} + \frac{c_3}{4}\left(\frac{1}{2}-2\ln 2\right) = \frac{c_3}{8},$$ which is equivalent to
	\begin{equation} \label{eq5}
	c_{-1} - 2c_3\ln 2  = 0.
	\end{equation}
	Again, if we take $n=1$ in Equation $(\ref{eq2})$ and apply Lemma \ref{Mellin}, we obtain
	$$ 8\widehat{f_{-1}}(7)=6c_3\widehat{r^2}(8)+2\widehat{f_{-1}}(3),$$
	which, using Equation $(\ref{eq4})$, is equivalent to
	\begin{equation}\label{c_{-1}}
	 6\left[c_{-1}\widehat{1}(6) + \widehat{1}(6)\widehat{f_*}(6)\right] = 2\left[c_{-1}\widehat{1}(2) +  \widehat{1}(2)\widehat{f_*}(2)\right] + \frac{6c_3}{10}
	 \end{equation}
	with
	$$ \widehat{f_*}(6)= -4c_3\int_{0}^{1}r^2\frac{1-r^6}{1-r^4}r^5dr = -4c_3\Big(\frac{31}{60}-\frac{\ln 2}{2}\Big),  $$
	and
	$$ \widehat{f_*}(2)= -4c_3 \int_{0}^{1} r^2 \frac{1-r^6}{1-r^4}rdr = c_3 \Big(2\ln 2 - \frac{8}{3}\Big). $$
	After substituting $ \widehat{f_*}(6) $ and $ \widehat{f_*}(2) $ in Equation $(\ref{c_{-1}})$ and simplifying, we obtain
	\begin{equation} \label{eq6}
	2c_{-1} + \left(2\ln 2-\frac{8}{3}\right)c_3 = 0.
	\end{equation}
	But it is easy to see that equations (\ref{eq5}) and (\ref{eq6}) are both satisfied if and only if $ c_{-1}=c_{3}=0 $, because the determinant
	$$ 	\begin{vmatrix}
		1 & -2\ln 2 \\
		1 & 2 \ln 2 - \frac{8}{3}
		\end{vmatrix} \neq0.$$
	Therefore $f_{-1}(r) =0= f_{3}(r)$. Now in $(\ref{sum})$, the terms in $z^{n-3}$ come only from the product of $T_{e^{-5i\theta}f_{-5}}$ with $T_{z^2}$ because $T_{e^{-i\theta}f_{-1}}=0$, ans so $T_{e^{-5i\theta}f_{-5}}$ commutes with $T_{z^2}$, which by Theorem 2 is possible only if $f_{-5}(r)=0$. Repeating the same argument,  we show that $ f_{-1 + 4k}(r)=0, \ \forall k\leq-1.$
	
	Next, we shall prove that $ f_{-3}(r)= 0=f_{1}(r)$, and consequently $f_{-3+4k}(r)=0$ for all $k\leq -1$.
	In $(\ref{sum})$, the terms in $ z^{n+3} $ come from the product of $ T_{e^{i\theta}f_1} $ with $ T_{z^2} $ only, hence $ f_1 (r)= c_1r$ for some constant $c_1$. Now, the terms in $ z^{n-1}  $ come from the product of $ T_{{e^-3i\theta}f_{-3}} $ with $ T_{z^2} $ and the product of $T_{e^{i\theta}f_1}=c_1T_z$ with $T_{\bar{z}^2}$. Therefore we must have
		$$ c_1 T_z T_{\bar{z}^2} (z^n) + T_{e^{-3i\theta}f_{-3}} T_{z^2}(z^n) = c_1  T_{\bar{z}^2}T_z (z^n) + T_{z^2}T_{e^{-3i\theta}f_{-3}}(z^n) \ \forall n\geq 0,$$
		which, using Lemma \ref{Mellin}, implies
	$$ c_1 \frac{n-1}{n+1} z^{n-1} + 2n \widehat{f_{-3}}(2n+3) z^{n-1} = c_1 \frac{n}{n+2} z^{n-1} + 2(n-2) \widehat{f_{-3}}(2n-1) z^{n-1},\ \forall n\geq 3. $$
	It follows that,
	$$ 2n \widehat{f_{-3}}(2n+3) - 2(n-2) \widehat{f_{-3}}(2n-1) = c_1\Big[\frac{n}{n+2} - \frac{n-1}{n+1} \Big], \forall n\geq 3. $$
	Applying Lemma \ref{complex}, the previous equation becomes
	 $$ (z+4)\widehat{f_{-3}}(z+7) - z \widehat{f_{-3}}(z+3) = c_1\Big[\frac{z+4}{z+8} - \frac{z+2}{z+6} \Big], \textrm{ for } \Re z\geq 2, $$
	 or
	 \begin{equation}  \label{eq7}
	 (z+4)\widehat{r^3f_{-3}}(z+4) - z \widehat{r^3f_{-3}}(z) = -4c_1\Big[\frac{1}{z+8} - \frac{1}{z+6} \Big], \textrm{ for } \Re z\geq 2.
	\end{equation}
	Here we let $\displaystyle{f_* (r) = -4c_1 r^6 \frac{1-r^2}{1-r^4}}$. Then
	\begin{eqnarray*}
		\widehat{f_*}(z+4)-\widehat{f_*}(z)&=&\widehat{r^4f_*}(z)-\widehat{f_*}(z)\\
		&=&\widehat{(r^4-1)f_*}(z)\\
		&=& 4c_1\int_{0}^{1} \frac{(1-r^4)r^6(1-r^2)}{(1-r^4)} r^{z-1}dr\\
		&=& -4c_1 \Big[\frac{1}{z+8}- \frac{1}{z+6}\Big].
	\end{eqnarray*}
	Thus, Equation (\ref{eq7}) can be written as
	$$ F(z+4)-F(z)= \widehat{f_*}(z+4)- \widehat{f_*}(z)$$
	where $ F(z)=z \widehat{r^3f_{-3}}(z) $.
	So using Lemma 3, we obtain
	$$ z \widehat{r^3f_{-3}}(z) = c_{-3} + \widehat{f_*}(z), \textrm{ for some constant } c_{-3}. $$
	Now, the Mellin convolution property $(2)$ implies
	$$ \widehat{r^3f_{-3}}(z) = c_{-3}\widehat{1}(z) + \widehat{1}(z)\widehat{f_*}(z) = c_{-3}\widehat{1}(z) + \widehat{(1 *_M f_*)}(z), $$
	with
	\begin{eqnarray*}
	\widehat{(1 *_M f_*)}(z) &=& \int_{r}^{1} 1(\frac{r}{t}) f_*(t)\frac{dt}{t} \\ &=& -4c_1\int_{r}^{1}\frac{t^6(1-t^2)}{1-t^4}\frac{dt}{t} \\ &=& -2c_1\Big[\Big(\ln 2 -\frac{1}{2}\Big) + \Big(3r^2-\frac{r^4}{2}-\ln (1+r^2)\Big)\Big].
	\end{eqnarray*}
	Hence, 	$$ \widehat{r^3f_{-3}}(z) = c_{-3}\widehat{1}(z) -2c_1\Big[\Big(\ln 2 -\frac{1}{2}\Big) + \Big(3r^2-\frac{r^4}{2}-\ln (1+r^2)\Big)\Big], $$
	or
	$$ f_{-3} (r) = \frac{c_{-3}}{r^3}- c_1\Big[\Big(\frac{2\ln 2 -1}{r^3}\Big) + \Big(\frac{6}{r}-r-\frac{\ln (1+r^2)}{r^3}\Big)\Big]. $$ But clearly $f_{-3}$ is not in $ L^1([0,1],rdr) $ unless $ c_{-3}=0 $ and $ c_{1}=0 $, and therefore $ f_{-3}(r)=0 $ and $ f_1 (r)= 0 $. Now, in $(\ref{sum})$, the terms in $z^{n-5}$ come from the product of $T_{e^{-7i\theta}f_{-7}}$ with $T_{z^2}$ only because $T_{e^{-3i\theta}f_{-3}}=0$. Thus $T_{e^{-7i\theta}f_{-7}}$ must commute with $T_{z^2}$, and hence by Theorem 2 we have that $f_{-7}(r)=0$. Similarly, we prove that $f_{-3+4k}(r)=0$ for all $k\leq -2$.
	
	Now going back to Equation $(\ref{sum})$, the terms in $ z^{n+4} $ come only from the product of $ T_{e^{i2\theta}f_2} $ with $ T_{z^2} $, and so these operators must commute. Thus, Theorem 2 implies that $ f_2 (r)= c_2 r^2$ for some constant $c_2$, i.e., $T_{e^{2i\theta}f_2}=c_2T_{z^2}$.  Similarly, the terms in $ z^n $ come from the products of $ T_{e^{-2i\theta}f_{-2}} $ with $ T_{z^2} $ and from the product of $ T_{\bar{z}^2} $ with $ T_{e^{i2\theta}f_2} = c_2T_{z^2}$, and hence by equality we have
	$$ c_2 T_{z^2} T_{z^{-2}}(z^n) + T_{e^{-2i\theta}f_{-2}}T_{z^2}(z^n) = c_2  T_{\bar{z}^2}T_{z^2}(z^n) + T_{z^2}T_{e^{-2i\theta}f_{-2}} (z^n), \forall n\geq 0.$$
	Thus, Lemma \ref{Mellin} implies
	$$ c_2 \frac{n-1}{n+1}  + 2(n+1)\widehat{f_{-2}}(2n+4) = c_2 \frac{n+1}{n+3} + 2(n-1)\widehat{f_{-2}}(2n), \forall n\geq 2. $$
	Applying Lemma \ref{complex} to the previous equation, we obtain
	$$ 2(z+1)\widehat{f_{-2}}(2z+4) -  2(z-1)\widehat{f_{-2}}(2z) =  c_2\Big[ \frac{z+1}{z+3} - \frac{z-1}{z+1} \Big], \textrm{ for } \Re z\geq 2, $$
	which can be rewritten as
	$$ F(z+2)-F(z) = G(z+2)-G(z), \textrm{ for } \Re z\geq 2 $$
	with $ F(z)= 2(z-1)\widehat{f_{-2}}(2z) $ and $ \displaystyle{G(z)= c_2\frac{z-1}{z+1}}$.
	So by Lemma \ref{periodic}, we have
	$$ F(z)= c_{-2} + G(z), \textrm{ for some constant } c_{-2}.$$
	Hence
	$$ 2(z-1)\widehat{f_{-2}}(2z) = c_{-2} + c_2\frac{z-1}{z+1}, $$
	or
	\begin{eqnarray*}
		\widehat{f_{-2}}(2z) &=& \frac{c_{-2}}{2(z-1)} + \frac{c_2}{2(z+1)} \\ &=& c_{-2} \widehat{r^{-2}}(2z) + c_2 \widehat{r^2}(2z),
	\end{eqnarray*}
and therefore
$$ f_{-2}(r)= c_{-2}r^{-2} + c_2 r^2. $$
But clearly in the expression of $f_{-2}$, the term $c_{-2}r^{-2}$ is not in $L^1 ([0,1], rdr) $ unless $ c_{-2}=0 $, and so in this case $ f_{-2}(r)= c_2 r^2 $, i.e., $T_{e^{-2i\theta}f_{-2}}=c_2T_{\bar{z}^2}$.

Finally, in $(\ref{sum})$ the terms in $ z^{n-4}$  come from the product of $ T_{e^{-6i\theta}f_{-6}} $ with $T_{z^2} $ and  the product of $T_{\bar{z}^2} $ with $T_{e^{-2i\theta}f_{-2}} = c_2 T_{\bar{z}^2}$. Thus  by equality, we must have
$$ c_2 T_{\bar{z}^2} T_{\bar{z}^2} (z^n) +  T_{e^{-6i\theta}f_{-6}}T_{z^2} (z^n) = c_2 T_{\bar{z}^2} T_{\bar{z}^2} (z^n) + T_{z^2} T_{e^{-6i\theta}f_{-6}} (z^n), \forall n\geq 0, $$
which is equivalent to
$$ T_{e^{-6i\theta}f_{-6}}T_{z^2} (z^n) = T_{z^2} T_{e^{-6i\theta}f_{-6}} (z^n), \forall n\geq 0. $$
The previous equation tells us that $ T_{e^{-6i\theta}f_{-6}} $ commutes with $T_{z^2} $, and by Theorem 2, this is possible only if $ f_{-6}(r)=0 $. Now in $(\ref{sum})$, the terms in $z^{n-8}$ come from the product of $T_{e^{-10i\theta}f_{-10}}$ with $T_{z^2}$ only because $T_{e^{-6i\theta}f_{-6}}=0$. So $T_{e^{-10i\theta}f_{-10}}$ must commute with $T_{z^2}$, and again Theorem 2 implies that $f_{-10}(r)=0$. Similarly, we prove that $ f_{-2+4k}(r)=0, \ \forall k\leq-2 $.
This completes the proof.
\end{proof}
\begin{rem}
\begin{itemize}	
\item[i)] It is easy to see through the proofs that our results remain true if the symbol $z^2+\bar{z}^2$ is replaced by any linear combination of $z^2$ and $\bar{z}^2$, i.e., $\alpha z^2+\beta\bar{z}^2$.
\item[ii)] If the polar decomposition the symbol $f$ is instead truncated below, i.e., $f(re^{i\theta})=\sum_{k=N}^{\infty}$ where $N$ is a negative integer, then the result remains true since one can pass to the adjoint.
\end{itemize}
\end{rem}

\end{document}